\documentclass[11pt]{amsart}
\usepackage{graphicx, color, epsfig}
\usepackage{amscd}
\usepackage{amsmath,amsthm,empheq}
\usepackage{amsfonts}
\usepackage{amssymb}
\usepackage{mathrsfs}
\usepackage{graphicx}
\usepackage[all]{xy}

\numberwithin{equation}{section}

\usepackage[latin1]{inputenc}

\usepackage[english]{babel}

\newtheorem{corollary}{Corollary}[section]
\newtheorem{definition}[corollary]{Definition}

\newtheorem{lemma}[corollary]{Lemma}

\newtheorem{remark}[corollary]{Remark}
\newtheorem{thm}[corollary]{Theorem}

\newfont{\sBlackboard}{msbm10 scaled 900}

\newcommand{\mylabel}[1]{\label{#1}
            \ifx\undefined\stillediting
            \else \fbox{$#1$}\fi }
\newcommand{\BE}{\begin{equation}}

\newcommand{\EEQ}{\end{equation}}
\newcommand{\rfb}[1]{\mbox{\rm
   (\ref{#1})}\ifx\undefined\stillediting\else:\fbox{$#1$}\fi}

\newfont{\Blackboard}{msbm10 scaled 1200}

\newfont{\roma}{cmr10 scaled 1200}

\newcommand{\bb}{\begin{equation}}
\newcommand{\bbb}{\end{equation}}

\newcommand{\mm}    {{\hbox{\hskip 0.5pt}}}

\newcommand{\bluff} {{\hbox{\raise 15pt \hbox{\mm}}}}
scaled\magstep2

\makeatletter
\def\section{\@startsection {section}{1}{\z@}{-3.5ex plus -1ex minus
    -.2ex}{2.3ex plus .2ex}{\large\bf}}
\makeatother

\begin{document}
\title[Brezis-Lieb Lemma]{A   Brezis-Lieb-type Lemma in Orlicz space}
\author[A. Bahrouni]{Anouar Bahrouni}
\address[A. Bahrouni]{Mathematics Department, University of Monastir,
Faculty of Sciences, 5019 Monastir, Tunisia} \email{\tt
bahrounianouar@yahoo.fr}
\begin{abstract}
In this work, we extend the well known Brezis-Lieb Lemma to the
Orlicz space.
\end{abstract}
 \maketitle
\section{Introduction and main result}
The Brezis-Lieb lemma, see \cite{1},  has major applications mainly
in calculus of variations, see \cite{b,l}. To describe a special
case of this theorem, suppose that $\Omega\subset \mathbb{R}^{N}$ is
a domain, $p>1$, $f(t)=|t|^{p}$ for $t\in \mathbb{R}$, $(u_{n})$ a
bounded sequence in $L^{p}(\Omega)$ and $u_{n} \rightarrow u \in
L^{p}(\Omega)$ a.e. on $\Omega$. If one denotes by
$\mathcal{F}:L^{p}(\Omega) \rightarrow L^{1}(\Omega)$ superposition
operator f generates, i.e., $\mathcal{F}(v)(x):=f((v)(x))$, then
$u\in L^{p}(\Omega) $ and
$$\mathcal{F}(u_{n})-\mathcal{F}(u_{n}-u) \rightarrow \mathcal{F}(u) \ \ \mbox{in} \ \ L^{1}(\Omega), \ \ \mbox{as} \ \ n\rightarrow +\infty.$$
The main purpose of this note is to give a similar result in the
Orlicz
space. \\
To introduce our main result more precisely, we first give some
basic definitions and properties concerning Orlicz space. We start
by recalling the definition of the well-known Orlicz functions.
\begin{definition}
$(1)$ $G:\mathbb{R}_{+}\rightarrow \mathbb{R}_{+}$ is called an
Orlicz
function if it has the following properties:\\
$(H_{1})$ $G$ is continuous, convex, increasing and $G(0)=0$.\\
$(H_{2})$ $G$ satisfies the $\Delta_{2}-$condition, that is there
exists $K>2$ such that
$$G(2x)\leq KG(x) \ \ \mbox{for all} \ \ x\in \mathbb{R}_{+}.$$
$(H_{3})$ $\displaystyle \lim_{x\rightarrow 0}\frac{G(x)}{x}=0.$\\
\\
$(2)$ given $G$ an Orlicz function, we define the complementary
function $G^{\ast}$ as
$$G^{\ast}(a)=\sup\{at-G(t), \ \ t>0\}.$$
\end{definition}
\begin{remark}
$(i)$From the above definition it is immediate that the following
Young-type inequality holds
\begin{equation}\label{young}
ab\leq G(a)+G^{\ast}(b) \ \ \mbox{for every} \ \ a,b\geq0.
\end{equation}
$(ii)$ It is shown in \cite{2} that the $\Delta_{2}$ condition is
equivalent to
\begin{equation}\label{delta}
\frac{G^{'}(a)}{G(a)}\leq \frac{p}{a}, \ \ \forall a> 0,
\end{equation}
for some $p>1$.
\end{remark}
Now, we are ready to state our main result.
\begin{thm}\label{brlb}
Let $\Omega\subset \mathbb{R}^{N}$ be an open domain and $G$ be an
Orlicz function. Moreover, we assume that $G$ is of class $C^{1}$ on
$\mathbb{R}_{+}.$ Suppose that $(u_{n})$ is bounded in
$L^{G}(\Omega)$\\ and $u_{n} \rightarrow u \in L^{G}(\Omega)$ a.e.
on $\Omega$. Then
\begin{equation}
\displaystyle \lim_{n\rightarrow
+\infty}\int_{\Omega}G(|u_{n}|)dx-\int_{\Omega}G(|u_{n}-u|)dx=\int_{\Omega}G(|u|)dx.
\end{equation}
\end{thm}
\section{Proof of Theorem \ref{brlb}}
First, we prove the following technical lemma.
\begin{lemma}\label{lemma}
Let $G$ be an Orlicz function and $G^{\ast}$ his complementary. Then:\\
$(i)$ $G(a+b)\leq \frac{K}{2}(G(a)+G(b))$ for every $a,b\geq 0.$\\
$(ii)$ There is $C>0$ such that for every $\epsilon \in (0,1)$ we
have
$$g(a)b\leq C(\epsilon G(a)+G(\frac{b}{\epsilon})) \ \ \mbox{for every} \ \ a,b\geq0,$$
where $g=G^{'}.$
\end{lemma}
\begin{proof}
$(i)$  Invoking conditions $(H_{1})$ and $(H_{2})$, we deduce that
\begin{align*}
G(a+b)&=G(2\frac{a+b}{2})\leq K G(\frac{a+b}{2})\\
&\leq \frac{K}{2}(G(a)+G(b)).
\end{align*}
$(ii)$ In what follows we show that
\begin{equation}\label{p}
G^{\ast}(g(t))\leq (p-1)G(t), \ \ \forall t>0,
\end{equation}
where $p>1$ is given by \eqref{delta}. Fix $t>0$. Let
$h(a):=g(t)a-G(a)$. Then, it is easy to see that $h(a)\leq h(t)$ for
every $a>0$. Therefore
$$G^{\ast}(t)=g(t)t-G(t).$$
Combining the above identity with \eqref{delta} we prove \eqref{p}.
Let $a,b >0$ and $\epsilon \in (0,1)$. Then, by \eqref{young} and
\eqref{p}, we infer that
\begin{align*}
g(a)b&= \epsilon g(a) \frac{b}{\epsilon}\leq G^{\ast}(\epsilon
g(a))+G(\frac{b}{\epsilon})\\
&\leq  \epsilon G^{\ast}( g(a))+G(\frac{b}{\epsilon})\\
&\leq \epsilon (p-1)G(a)+G(\frac{b}{\epsilon})\\
&\leq \max((p-1),1)(\epsilon G(a)+ G(\frac{b}{\epsilon})).
\end{align*}
This ends the proof.
\end{proof}
\textbf{Proof of Theorem \ref{brlb} completed}:\\ \\
Using the Taylor formula, for any fixed $x\in \Omega$, we have
$$G(|u_{n}|)=G(|u_{n}-u|)+g(\xi)(|u_{n}|-|u_{n}-u|),$$
where $\xi$ is a measurable function with values between
$|u_{n}(x)|$ and $|u_{n}(x)-u(x)|$. Therefore
\begin{align}\label{eq1}
|G(|u_{n}|)-G(|u_{n}-u|)|&\leq g(\xi) |u|\\
&\leq g(|u_{n}|+|u_{n}-u|)|u|\nonumber\\
&\leq g(|u|+2|u_{n}-u|)|u|\nonumber.
\end{align}
By Lemma \ref{lemma}, for fixed $\epsilon \in (0,1)$, we get
\begin{align*}
g(2|u_{n}-u|+|u|)|u| &\leq C( \epsilon
G(2|u_{n}-u|+|u|)+G(\frac{|u|}{\epsilon}))\\
&\leq \frac{C K^{2}}{2}( \epsilon
G(|u_{n}-u|)+G(|u|)+G(\frac{|u|}{\epsilon}))\\
&\leq \frac{C K^{2}}{2}( \epsilon G(|u_{n}-u|)+2
G(\frac{|u|}{\epsilon})).
\end{align*}
It follows, using Lemma \ref{lemma} and \eqref{eq1} , that
\begin{equation*}
|G(|u_{n}|)-G(|u_{n}-u|)|\leq \frac{C K^{2}}{2}( \epsilon
G(|u_{n}-u|)+2G(\frac{|u|}{\epsilon})).
\end{equation*}
Let
$$W_{\epsilon,n}=[|G(|u_{n}|)-G(|u_{n}-u|)-G(|u|)|-\epsilon \frac{C K^{2}}{2}
G(|u_{n}-u|)]_{+},$$ where $[a]_{+}=\max(a,0).$ As $n \rightarrow
+\infty$, $W_{\epsilon,n}(x)\rightarrow 0$ a.e. on $\Omega$. On the
other hand \begin{align*} |G(|u_{n}|)-G(|u_{n}-u|)-G(|u|)|& \leq
|G(|u_{n}|)-G(|u_{n}-u|)|+G(|u|)\\
&\leq (CK^{2}+1) G(\frac{|u|}{\epsilon})+ \epsilon \frac{C K^{2}}{2}
G(|u_{n}-u|).
\end{align*}
Therefore
$$W_{\epsilon,n}x\leq (CK^{2}+1) G(\frac{|u|}{\epsilon})\in L^{1}(\Omega).$$
Then, by Lebesgue convergence theorem, $\displaystyle
\int_{\Omega}W_{\epsilon,n}(x)dx \rightarrow 0$ as $n\rightarrow
+\infty.$ From the fact that
$$\displaystyle \int_{\Omega}|G(|u_{n}|)-G(|u_{n}-u|)-G(|u|)|dx\leq \int_{\Omega}(W_{\epsilon,n}(x)
+\epsilon \frac{C K^{2}}{2} G(|u_{n}-u|))dx, $$ we deduce that
$$\displaystyle \limsup_{n\rightarrow +\infty}\int_{\Omega}|G(|u_{n}|)-G(|u_{n}-u|)-G(|u|)|dx\leq \epsilon C^{'} K, $$
for some positive constant $C^{'}$. Letting $\epsilon \rightarrow0$
we complete the proof.


\begin{thebibliography}{99}
\bibitem{1} H. Brezis, E. Lieb, A relation between pointwise convergence functions and convergences of functionals, {\it Proc. Amer. Math. Soc.}, {\bf  88} (1983) 486-490.
\bibitem{b} H. Brezis, L. Nirenberg, Positive solutions of nonlinear elliptic
equations involving critical Sobolev exponents, {\it Commun Pur Appl
Math.} {\bf 36} (1983) 437-477. \bibitem{l} E. Lieb, Sharp constants
in the Hardy-Littlewood-Sobolev and related inequalities, {\it  Ann
Math.} {\bf 118} (1983) 349-374.
\bibitem{2} M.M. Rao, Z.D. Ren,
{\it Theory of Orlicz Spaces}, Marcel Dekker Inc., New York, 1991.
\end{thebibliography}
\end{document}